\documentclass[final,onefignum,onetabnum]{myclass}

\usepackage{amsfonts}
\usepackage{graphicx}
\usepackage{epstopdf}
\usepackage{algorithmic}
\usepackage{booktabs}
\ifpdf
  \DeclareGraphicsExtensions{.eps,.pdf,.png,.jpg}
\else
  \DeclareGraphicsExtensions{.eps}
\fi
\usepackage{dsfont}

\usepackage{amsmath,amssymb}
\usepackage{enumitem}
\setlist[enumerate]{leftmargin=.5in}
\setlist[itemize]{leftmargin=.5in}

\setlist[itemize]{topsep=2pt,itemsep=2pt,parsep=0pt}
\setlist[enumerate]{topsep=2pt,itemsep=2pt,parsep=0pt}

\newsiamremark{remark}{Remark}
\newsiamremark{hypothesis}{Hypothesis}
\crefname{hypothesis}{Hypothesis}{Hypotheses}
\newsiamthm{claim}{Claim}
\newsiamremark{fact}{Fact}
\crefname{fact}{Fact}{Facts}

\newcommand{\JS}[1]{{\color{teal} #1}}

\usepackage{tikz}

\newcommand{\GP}{\mathrm{GP}}

\headers{Last-Coverage Probabilities in Biased Walks}{K. Idnani}

\title{Last-Coverage Probabilities of Random Walks on Cycle Graphs for Symmetric, Biased, and Time-Dependent Walks}

\author{Kiyan Idnani\thanks{University of Pennsylvania, Philadelphia, PA, USA
  (\email{idnanik@sas.upenn.edu}).}}

\dedication{\small\textit{Project advisor: Jad Salem\thanks{United States Naval Academy, Annapolis, MD, USA.}}}

\usepackage{amsopn}

\ifpdf
\hypersetup{
  pdftitle={Last-Coverage Probabilities of Random Walks on Cycle Graphs},
  pdfauthor={K. Idnani}
}
\fi

\begin{document}

\maketitle

\begin{abstract}

This paper studies the last-coverage vertex of a nearest-neighbour random walk on a cycle: the vertex most likely to be the final one discovered. We show that the visited vertices always form a single contiguous arc of the ring, a property we call the Connected Arc Property. This lets the walk be replaced by a two-state time-inhomogeneous Markov chain, which we call the interval-expansion representation, in which each expansion step has a probability given by the classical gambler's ruin formula. Within this framework we prove that the last-coverage probability of node $k$ in an $n$-node cycle is $(1-\rho)\rho^{\,n-1-k}/(1-\rho^{\,n-1})$, where $\rho = (1-p)/p$ and $p$ is the clockwise transition probability. Taking $p \to 1/2$ recovers the uniform distribution, giving an independent derivation of a classical result. For a walk with a deterministic time-varying bias we give an exact series representation for the last-coverage probabilities, together with a finite truncation that serves as a computable, error-bounded approximation. All results are validated by Monte Carlo simulation, and a sinusoidal schedule is examined in detail. Finally, we compare generalized Petersen graphs with Erd\H{o}s--R\'enyi graphs of matching size, finding that structural symmetry is associated with more uniform last-coverage distributions.
\end{abstract}

\section{Introduction}\label{introduction}

\subsection{Motivation and Background}\label{motivation-and-background}

Consider a security robot that patrols a circular corridor by repeatedly
choosing, at each decision point, whether to continue in the current
direction or to reverse course. In practice such a patrol is rarely
perfectly symmetric. As such, a central operational question arises: which location is
most likely to be reached last?

While the last-coverage problem can in principle be posed on any graph,
the cycle graph offers the simplest setting in which it can be solved
exactly, and it is the natural starting point before turning to more
general graphs, as we do in Section~\ref{open-problems-and-future-directions}.
The mathematical setting is the cycle graph $C_n$, a graph consisting of
$n$ vertices arranged in a ring, each connected by an edge to its two
immediate neighbours. A random walk on $C_n$ is a sequence of random
variables $\{X_i : i = 0, 1, \ldots\}$ such that $X_i$ is a vertex of
$C_n$ and $X_{i+1}$ is a neighbour of $X_i$ for all $i$. The cover time
$T_{cover}$ is the first moment at which every vertex has been visited.

The cover time of random walks on graphs has been studied extensively.
The work of Aleliunas, Karp, Lipton, Lov\'asz, and Rackoff
established that random walks cover all vertices efficiently
\cite{aleliunas1979random}, and the subject is now treated
comprehensively in Aldous and Fill and in Levin and Peres
\cite{aldous2002reversible, levin2026markov}. For the symmetric walk on $C_n$, the expected
cover time is exactly $\binom{n}{2}=n(n-1)/2$, and hence of order $n^2$
steps \cite{aleliunas1979random,levin2026markov,aldous2002reversible}.

The present paper focuses on a complementary and less-studied object:
the identity of the last vertex reached. This is the vertex $L$ that the
walk steps onto at the moment of complete coverage, that is, the vertex
$X_T$ discovered for the first time at time $T_{cover}$. The central goal is to
determine $P(L = k)$ for each non-starting vertex $k$. Last-coverage
distributions have been studied for unbiased random walks on cycles and
on complete graphs
\cite{lovasz1993note,buchanan2022last,aleliunas1979random}, and appear
implicitly in applied settings such as robotic coverage
\cite{choset2001coverage}, token-passing protocols
\cite{donnan2005token}, and cover-time bounds for general Markov chains
\cite{avin2007cover}. However, exact last-coverage results for biased or
time-varying walks, even on the simplest graphs, remain sparse, which is
the gap this paper addresses.

Rather than treating the unbiased walk as the base case and the biased
walk as a generalisation of it, we solve the
biased walk directly, and then recover the unbiased result as a limiting
case of the biased formula. This yields a second, independent derivation
of the classical uniform last-coverage distribution, and it is the
natural order of presentation once the interval-expansion framework is
in hand, since the biased formula falls from the unbiased case.

\subsection{Research Questions}\label{research-questions}

This paper addresses three related questions, each corresponding to a
different type of walk on $C_n$:

\begin{itemize}
\item
  \emph{Biased walk.} When one direction is consistently more probable
  than the other, how does this directional preference shift the
  last-coverage distribution, and can an exact formula be derived?
\item
  \emph{Unbiased walk.} Can the known uniform last-coverage
  distribution be recovered as a limiting case of the biased-walk
  formula, confirming the validity of that formula and giving an
  alternative proof of the classical result?
\item
  \emph{Time-varying bias.} When the directional preference follows a
  deterministic schedule that changes at each step, can the exact
  last-coverage probabilities still be computed, and if not in closed
  form, can they be approximated with a controllable error?
\end{itemize}

\subsection{Main Contributions}\label{main-contributions}

\begin{itemize}
\item
  Theorem~\ref{Theorem-3.1} (Connected Arc Property): the set of vertices
  visited by the walk always forms a single contiguous stretch of the
  ring. This structural fact is the foundation for all that follows.
\item
  Theorem~\ref{theorem-3.6} (Interval Expansion Representation): the
  coverage process is fully described by recording the direction of each
  new discovery as the visited stretch grows one step at a time, and the
  probability of each such directional sequence is a product of
  gambler's ruin exit probabilities.
\item
  Theorem~\ref{Theorem-4.1} (Biased Walk Formula): for fixed
  $p \neq 1/2$, representing the boundary expansions as a two-state
  time-inhomogeneous Markov chain -- whose states record whether the
  most recent discovery was to the left or to the right, and whose
  transition probabilities evolve as the arc grows -- yields an explicit
  closed form for the last-coverage probability of every vertex.
\item
  Lemma~\ref{Lemma-4.2} (Continuity at $p = 1/2$): $P_p(L=k)$, viewed as
  a function of the bias $p$, is continuous at $p = 1/2$, even though
  the closed form of Theorem~\ref{Theorem-4.1} is only directly defined
  for $p \neq 1/2$.
\item
  Corollary~\ref{Corollary-4.3} (Unbiased Walk as a Limiting Case):
  combining the two previous results, the $p \to 1/2$ limit of the
  biased formula recovers the classical uniform distribution
  $P(L=k)=1/(n-1)$ for $k\in\{1,\ldots,n-1\}$, independently confirming
  \cite{aleliunas1979random,buchanan2022last,lovasz1993note}.
\item
  Corollary~\ref{corollary-4.5} (Exact Series and Approximation
  Formula): for any deterministic time-varying bias, the last-coverage
  probabilities are given exactly by a convergent but generally infinite
  series over admissible expansion sequences. Truncating at a finite
  horizon gives a computable approximation whose error is bounded
  explicitly and can be made arbitrarily small.
\end{itemize}

Section~\ref{background} covers the necessary background.
Section~\ref{the-interval-expansion-framework} develops the
interval-expansion framework. Section~\ref{last-coverage-results}
derives the three main results, treating the biased walk first, the
unbiased walk second as a limiting case, and the time-varying walk
third. Section~\ref{monte-carlo-simulation} presents Monte Carlo
simulations, and Section~\ref{applications-open-problems-and-conclusion}
discusses applications and open problems.

\section{Background}\label{background}

\subsection{Cycle Graphs and Walk Models}\label{cycle-graphs}

The cycle graph $C_n$ consists of $n$ vertices labelled
$0, 1, \ldots, n - 1$ arranged in a ring, with vertex $i$ adjacent to
vertices $i - 1$ and $i + 1$, arithmetic taken modulo $n$. The
degenerate cases $n = 1$ and $n = 2$ are excluded throughout.

\begin{center}
\begin{tikzpicture}[scale=0.85]

\tikzstyle{vertex}=[draw, circle, color = black, text opacity = 1, inner sep = 1.4pt]

\foreach \s in {0,...,4}{
    \node[vertex] (v\s) at ({90 + 360/5 * (-\s)}:1.5cm){$\s$};
}
\foreach \s [evaluate = \r using {int(mod(\s+1,5))}] in {0,...,4}{
    \draw (v\s) -- (v\r);
}

\end{tikzpicture}
\end{center}

\begin{definition}[Connected arc]
A \emph{connected arc} in $C_n$ is a set of vertices
$S = \{a, a+1, \ldots, b\} \pmod{n}$ for some $a, b$; that is, a
contiguous, gap-free stretch of the cycle. For example, the vertices
$\{3, 4, 5, 6\}$ on a ten-vertex ring form a connected arc.
\end{definition}

A random walk on $C_n$ starts at vertex $0$ and at each step moves to
one of its two neighbours, chosen according to a specified probability
rule. This paper analyses two models.

In the \emph{biased (constant drift) walk}, each step goes clockwise
with probability $p$ and counterclockwise with probability $q = 1 - p$,
where $p$ is a fixed constant in $(0,1)$, the same at every step and
every vertex. When $p > 1/2$ the walk drifts clockwise and when
$p < 1/2$ it drifts counterclockwise. The special value $p = 1/2$,
treated in Section~\ref{the-unbiased-walk} as a limiting case of this
model, is called the \emph{unbiased} or \emph{symmetric} walk. For a symmetric walk neither
direction is favoured, the walk is reversible, and the uniform
distribution on the vertex set is its unique stationary distribution.

In the \emph{time-varying drift walk}, the step at time $t$ goes
clockwise with probability $p(t)$, where the schedule
$p(0), p(1), p(2), \ldots$ is deterministic and known in advance but
need not be constant. This model subsumes the previous ones.

In both models, both directions remain possible at every step, since
$p(t)$ lies strictly between $0$ and $1$ throughout. For the
constant-bias model this alone makes the walk an irreducible Markov
chain on the finite state space $C_n$, so it visits every vertex in
finite time with probability $1$. For the time-varying model this is no
longer automatic (a schedule could in principle decay toward $0$ or $1$
fast enough to make the walk directionally frozen), so throughout the
paper we additionally assume the schedule is uniformly bounded away from
the endpoints, $c \le p(t) \le 1-c$ for some constant $c>0$. This is
exactly the condition used in Theorem~\ref{theorem-3.6} and
Corollary~\ref{corollary-4.5} to guarantee almost-sure coverage in finite
time. For thorough treatments of random walks on graphs see
\cite{levin2026markov,aldous2002reversible}.

\subsection{Cover Time and the Last-Coverage Vertex}\label{cover-time-and-the-last-coverage-vertex}

The cover time is the first instance at which every vertex has been
visited,
\[
T_{\mathrm{cover}} = \min \{ t \geq 0 : \text{every vertex has been visited by time } t \},
\]
and the last-coverage vertex $L$ is the vertex first visited at time
$T_{\mathrm{cover}}$, that is, the vertex whose initial discovery
coincides with complete coverage. Since the walk starts at vertex $0$,
that vertex is visited immediately and can never be the last discovery,
so $L \in \{1,2,\ldots,n-1\}$. The primary objective of
Sections~\ref{the-interval-expansion-framework}--\ref{last-coverage-results}
is to determine $P(L = k)$ for each such $k$; the cover time itself
reappears in the analysis of time-varying walks in
Section~\ref{the-time-varying-bias-walk}.

\subsection{The Gambler's Ruin Formula}\label{the-gamblers-ruin-formula}

A classical result of probability theory, the gambler's ruin formula,
plays a central role throughout this paper. The setting is a
nearest-neighbour random walk on the integers $\{0, 1, \ldots, N\}$,
where the state starts at position $x$, moves right with probability $p$
at each step, and moves left with probability $q = 1 - p$. Play ends
when the state reaches either $0$ (ruin) or $N$ (success). Let $H(x)$
denote the probability of reaching $N$ before $0$. For completeness we
re-derive the biased case; the result is discussed in detail in Feller
and in Levin and Peres \cite{feller1971introduction,levin2026markov}.

\begin{theorem}[Biased gambler's ruin, \cite{feller1971introduction,levin2026markov}]
For a nearest-neighbour walk on $\{0, 1, \ldots, N\}$ with right-step
probability $p \neq 1/2$ and left-step probability $q = 1 - p$, the
probability of reaching $N$ before $0$, starting from position $x$, is
$H(x) = \bigl(1 - (q/p)^x\bigr) / \bigl(1 - (q/p)^N\bigr)$.
\end{theorem}

\begin{proof}
Conditioning on the first step gives the recurrence
$H(x) = p H(x + 1) + q H(x - 1)$, subject to $H(0) = 0$ and $H(N) = 1$,
since from position $x$ the state moves right with probability $p$ or
left with probability $q$. Rearranging gives the second-order linear
recurrence with constant coefficients
$p\,H(x+1) - H(x) + q\,H(x-1) = 0$. Substituting the trial solution
$H(x) = \lambda^x$ yields the characteristic equation
$p\lambda^2 - \lambda + q = 0$, which factors as
$(\lambda - 1)(p\lambda - q) = 0$, with distinct roots $\lambda = 1$ and
$\lambda = q/p$; the general solution is therefore
$H(x) = A + B (q/p)^x$. Applying $H(0) = 0$ gives $A = -B$, and
$H(N) = 1$ then gives $B\bigl(1 - (q/p)^N\bigr) = 1$. Substituting,
\[
H(x) = \frac{1 - (q/p)^x}{1 - (q/p)^N}. \qed
\]
\end{proof}

For the symmetric case $p = q = 1/2$ this formula cannot be obtained by
direct substitution, since setting $p = q$ yields the indeterminate form
$0/0$; taking the limit $p \to 1/2$ (equivalently $q/p \to 1$) via
L'H\^opital's rule gives $H(x) = x/N$. Lemma~\ref{Lemma-4.2} carries out
exactly this limit for the last-coverage probability itself, not just
for a single exit probability.

These formulas underlie all subsequent analysis. Given a visited arc,
the state moves within the arc until it exits through one of its two
endpoints and thereby discovers a new vertex. The probability that this
discovery occurs at the right boundary rather than the left is therefore
an instance of gambler's ruin, with the current position of the state
corresponding to $x$ and the two endpoints of the arc acting as
absorbing barriers.

\section{The Interval-Expansion Framework}\label{the-interval-expansion-framework}

This section develops the central analytical tool of the paper. Section~\ref{the-connected-arc-property}
proves the Connected Arc Property, the structural fact that the visited
set always forms a single contiguous arc.
Sections~\ref{boundary-counters}--\ref{exit-probabilities} use it to
reduce the coverage process to a sequence of left and right boundary
expansions, each governed by a gambler's ruin exit probability,
culminating in Theorem~\ref{theorem-3.6}.

\subsection{The Connected Arc Property}\label{the-connected-arc-property}

\begin{theorem}[Connected Arc Property] \label{Theorem-3.1}
Let $(X_t)$ be a nearest-neighbour random walk on $C_n$ starting at
vertex $0$. For every time $t \geq 0$, the set of vertices visited up to
time $t$, denoted $V_t$, forms a connected arc of the cycle. That is,
$V_t$ is either a proper sub-arc (a path) or all of $C_n$; equivalently,
the visited vertices always constitute a single contiguous stretch of
the ring.
\end{theorem}

\begin{proof}
We proceed by induction on $t$. At time $0$ only vertex $0$ has been
visited, so $V_0 = \{0\}$, a single vertex and hence a connected arc.

Suppose that at some time $t \geq 0$ the visited set $V_t$ forms a
connected arc with left endpoint $v_L$ and right endpoint $v_R$. Because
the walk is nearest-neighbour, $X_{t+1}$ must be a neighbour of $X_t$.
If $X_{t+1}$ belongs to $V_t$, then $V_{t+1} = V_t$, which is already a
connected arc. Otherwise $X_{t+1}$ is unvisited. Since $V_t$ is a
connected arc, every vertex in its interior (every vertex other than
$v_L$ and $v_R$) has both of its cycle-neighbours already inside $V_t$.
Consequently no interior vertex is adjacent to an unvisited vertex, and
a new discovery can occur only by stepping off one of the two endpoints.
Since $X_t$ belongs to $V_t$ and $X_{t+1}$ is unvisited, $X_t$ must be
either $v_L$ or $v_R$, and $X_{t+1}$ is the unique unvisited neighbour
of that endpoint. The set $V_{t+1} = V_t \cup \{X_{t+1}\}$ is therefore
obtained by appending one vertex to an endpoint of the arc, which
preserves the connected arc structure. In both cases $V_{t+1}$ is a
connected arc, so by induction the property holds for all $t \geq 0$.
\end{proof}

\begin{remark}
The argument uses only that the walk is nearest-neighbour on $C_n$; it
makes no reference to the step probabilities, so
Theorem~\ref{Theorem-3.1} holds unchanged for the biased and
time-varying walks of Section~\ref{cycle-graphs}. This is what allows
the same Connected Arc Property to underlie every result in
Section~\ref{last-coverage-results}. The inductive step also uses
$n \geq 3$: only then does an interior vertex of a proper sub-arc have
both cycle-neighbours inside the arc, making the two candidates for an
unvisited $X_{t+1}$ exhaustive.
\end{remark}

Theorem~\ref{Theorem-3.1} has four immediate consequences. First, the
visited region at time $t$ is completely described by two non-negative
integers, and no further information about the walk history is needed.
Second, new discoveries can occur only at the two endpoints of the
current arc. Third, coverage is complete precisely when the arc has
grown to encompass all $n$ vertices. Fourth, and most consequentially,
the coverage process admits a reduced description in which only the two
boundary positions matter, which collapses the infinitely many sample
paths with $L=k$ into finitely many expansion sequences. The remainder
of this section makes each of these precise.

\subsection{Boundary Counters and Expansion Sequences}\label{boundary-counters}

Given the random walk $(X_t)$ starting at vertex $0$, define two
boundary counters: $A_t$, the number of vertices discovered to the left
of the starting point, and $B_t$, the number discovered to the right. By
Theorem~\ref{Theorem-3.1} the visited region at time $t$ is the arc from
position $-A_t$ to position $+B_t$ (taken modulo $n$), which contains
$A_t + B_t + 1$ vertices including the starting vertex. Full coverage
occurs when $A_t + B_t = n - 1$, and each new discovery increments
exactly one of $A_t$ or $B_t$ by exactly $1$.

\label{expansion-sequences}%
Since each discovery extends the arc in exactly one direction, the
entire coverage history can be recorded as an ordered sequence of $n-1$
directions,
\[
(Y_1,\ldots,Y_{n-1}), \qquad \text{where } Y_i \in \{L,R\} \text{ for all } i,
\]
where $L$ corresponds to a counterclockwise move and $R$ to a clockwise
move. The last-coverage vertex is determined entirely by this sequence:
if the final expansion $Y_{n-1}$ is $R$, the last vertex discovered is
the one immediately to the right of the right boundary at step $n-2$,
and if $Y_{n-1}$ is $L$, it is the one immediately to the left of the
left boundary. Each possible last vertex $k$ therefore corresponds to a
specific set of direction sequences.

\subsection{Exit Probabilities}\label{exit-probabilities}

Between consecutive discoveries, the state moves within the current arc
until it steps off one end. The probability that it exits through the
right end is the exit probability $r(a, b, x)$, where $a$ and $b$ are
the current boundary sizes and $x$ is the state's position within the
arc. This is a direct application of the gambler's ruin formula of
Section~\ref{the-gamblers-ruin-formula} to an interval of length
$a+b+1$ with the state at a re-centred position, the two ends of the arc
serving as absorbing barriers \cite{levin2026markov,feller1971introduction}.
For the biased walk with ratio $\rho = q/p$, and for the unbiased walk
respectively,
\[
r(a, b, x) = \frac{1 - \rho^{\,x+a+1}}{1 - \rho^{\,a+b+2}},
\qquad\qquad
r(a, b, x) = \frac{x + a + 1}{a + b + 2}.
\]

\subsection{The Central Theorem}\label{the-central-theorem}

With the exit probabilities in hand, the coverage process can be
expressed exactly as a product of these probabilities over the sequence
of boundary expansions.

\begin{theorem}[Interval Expansion Representation] \label{theorem-3.6}
Let $(X_t)$ be a random walk on $C_n$ of the biased walk type from
Section~\ref{cycle-graphs} (a homogeneous nearest-neighbour walk with
fixed clockwise probability $p$) starting at vertex $0$, with
$p \in (0,1)$ satisfying $c \leq p \leq 1-c$ for some constant $c > 0$.
(This bound is automatic for any fixed $p\in(0,1)$, taking
$c = \min(p,1-p)$; it is stated in this uniform form because the
identical statement, with $p$ replaced by a schedule $p(t)$, underlies
the time-varying analysis of Section~\ref{the-time-varying-bias-walk},
where the bound is a genuine assumption.) Let $L$ denote the
last-coverage vertex, and for each non-starting vertex $k$ let $H_k$
denote the set of expansion sequences
$(y_1,\ldots,y_{n-1}) \in \{L,R\}^{n-1}$ in which vertex $k$ is the
final discovery. Then
\[
P(L = k) = \sum_{h \in H_k} P(h),
\qquad
P(h) = \prod_{i=1}^{n-1}
r_p\!\left(a^h_i,\,b^h_i,\,x^h_i\right)^{\mathbf{1}\{y_i = R\}}
\bigl(1 - r_p\!\left(a^h_i,\,b^h_i,\,x^h_i\right)\bigr)^{\mathbf{1}\{y_i = L\}},
\]
where $(a^h_i, b^h_i, x^h_i)$ is the state of the arc and walker
immediately before the $i$-th discovery under sequence $h$, and
\[
r_p(a,b,x)
= \frac{1-(q/p)^{\,x+a+1}}{1-(q/p)^{\,a+b+2}},
\qquad q = 1-p,
\]
with the convention $r_{1/2}(a,b,x) = (x+a+1)/(a+b+2)$ in the symmetric
case.
\end{theorem}

\begin{proof}
\textbf{Step 1} (Arc characterisation). By Theorem~\ref{Theorem-3.1},
the visited region is always a connected arc determined by the pair
$(A_t, B_t)$, and coverage is complete when $A_t + B_t = n - 1$.

\textbf{Step 2} (Discoveries at endpoints only). A new vertex can be
discovered only by stepping off one end of the current arc. Each
discovery therefore increments exactly one of $A_t$ or $B_t$ by $1$, and
the $n-1$ discovery events correspond exactly to a sequence of $n-1$
left or right expansions $(Y_1,\ldots,Y_{n-1})$.

\textbf{Step 3} (Almost sure termination). Since $p \in (0,1)$, the walk
$(X_t)$ is an irreducible Markov chain on the finite state space $C_n$:
every vertex is reachable from every other in finitely many steps with
positive probability. An irreducible finite-state Markov chain visits
every state infinitely often, and in particular reaches every vertex in
finite time, with probability $1$. Consequently exactly $n-1$ discovery
events occur almost surely, and the expansion sequence is well defined
for almost every realisation. (For the time-varying model of
Section~\ref{the-time-varying-bias-walk}, the same conclusion requires
the uniform bound $c \leq p(t) \leq 1-c$ and the argument given in
Step~2 of the proof of Corollary~\ref{corollary-4.5}.)

\textbf{Step 4} (Probability of an expansion sequence). At each
expansion step $i$, the walker wanders within the current arc until it
exits at one end, with probability $r_p(a^h_i, b^h_i, x^h_i)$ of exiting
to the right and $1 - r_p(a^h_i, b^h_i, x^h_i)$ to the left. Since these
exit events at successive steps are conditionally independent given the
arc state, the probability of a specific sequence
$h = (y_1,\ldots,y_{n-1})$ is the stated product.

\textbf{Step 5} (Summation). For each non-starting vertex $k$, the law
of total probability over the disjoint sequences in $H_k$ gives
$P(L = k) = \sum_{h \in H_k} P(h)$.
\end{proof}

\begin{remark}[Consistency check]
Since exactly one vertex must be discovered last, the sets $H_k$ are
disjoint and their union is the full set $\{L,R\}^{n-1}$, so
$\sum_{k=1}^{n-1} P(L=k) = \sum_{h \in \{L,R\}^{n-1}} P(h) = 1$, as it
must be. The starting vertex is visited at time $0$, so $H_0$ is empty
and $P(L=0) = 0$.
\end{remark}

\begin{remark}[Size of $H_k$ and computational cost] \label{remark-Hk}
If vertex $k$ is the last discovery, then among the first $n-2$
expansion steps there must be exactly $n-k-1$ left expansions and $k-1$
right expansions: the arc's right boundary must reach vertex $k-1$ and
its left boundary must reach vertex $k+1$ the other way around the ring
before the final step, wherever it comes from, closes the gap at $k$.
Since each expansion simply extends one boundary by one vertex, every
ordering of these $n-2$ expansions is admissible, so
$|H_k| = \binom{n-2}{k-1} = \binom{n-2}{n-k-1}$. Consequently
$\sum_{k=1}^{n-1} |H_k| = 2^{n-2}$, and direct enumeration costs
$O(2^n)$ time. Theorem~\ref{Theorem-4.1} below gives a closed form that
avoids this exponential cost entirely, and
Section~\ref{the-time-varying-bias-walk} develops an analogous
approximation for the time-varying case.
\end{remark}

Theorem~\ref{theorem-3.6} thus transforms a problem about a random walk
on a ring into a structured calculation involving only the exit
probabilities and the combinatorics of expansion sequences.
Section~\ref{last-coverage-results} carries out this calculation for
each walk model in turn.

\begin{figure}[ht]
    \centering
    \includegraphics[width=0.8\textwidth]{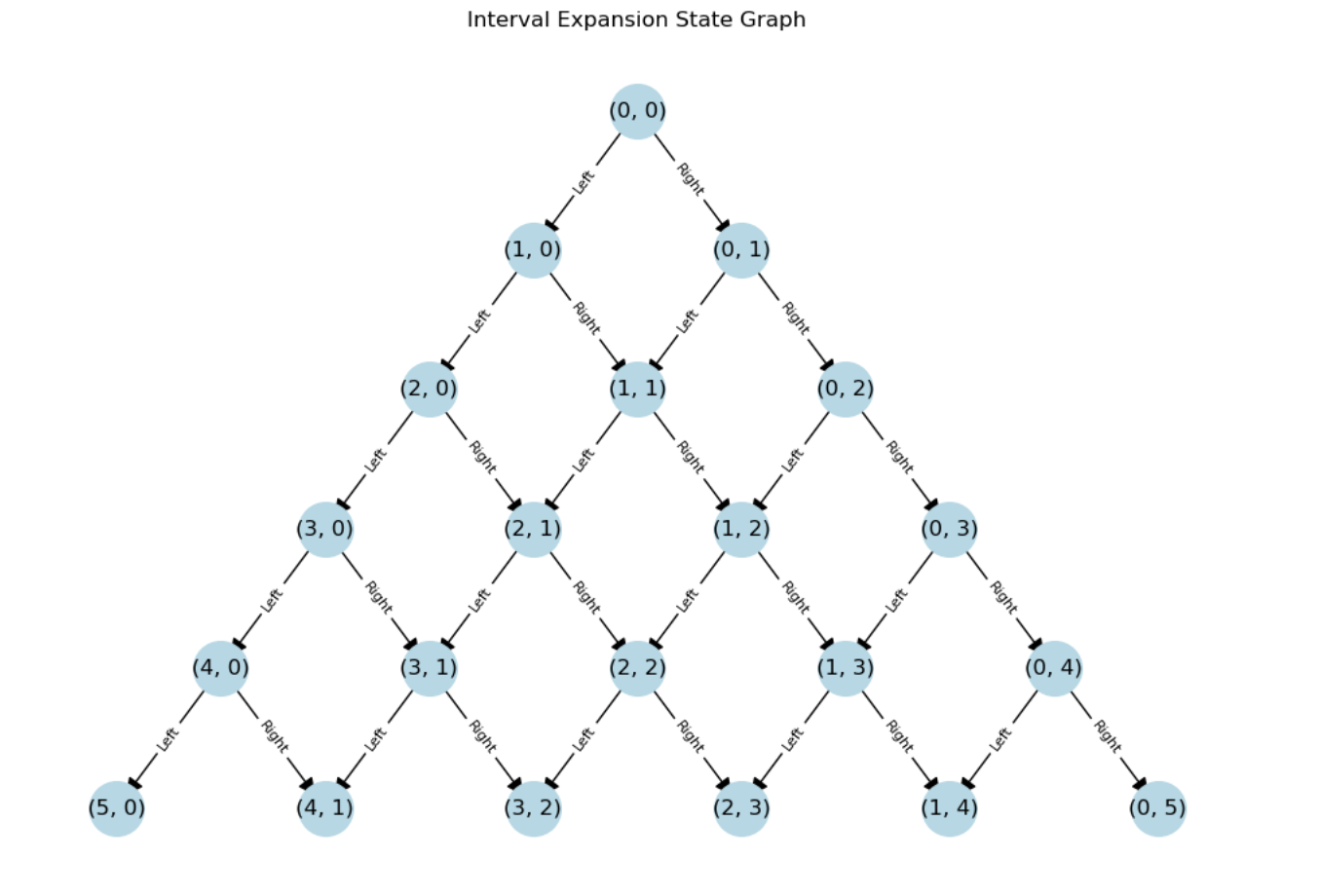}
    \caption{State-space diagram for the expansion sequences, in which each node represents a pair $(a,b)$, the current left and right boundary sizes. Directed edges represent left and right expansion events labelled with the corresponding exit probability.}
    \label{fig:interval-expansion}
\end{figure}

\section{Last-Coverage Results}\label{last-coverage-results}

This section applies the interval-expansion framework to the three walk
models. Section~\ref{the-biased-walk} treats the biased walk, deriving
an explicit closed form by representing the boundary expansions as a
two-state time-inhomogeneous Markov chain.
Section~\ref{the-unbiased-walk} recovers the known uniform distribution
as the $p \to 1/2$ limit of that formula, a limit justified by a
continuity lemma proved along the way.
Section~\ref{the-time-varying-bias-walk} treats the time-varying walk,
establishing an exact series representation together with a computable,
error-controlled approximation.

\subsection{The Biased Walk}\label{the-biased-walk}

At every step the state moves clockwise with probability $p$ and
counterclockwise with probability $q = 1 - p$, where $p$ is a fixed
constant in $(0,1)$, not necessarily $1/2$, the same at every step and
every position. Writing $\rho = q/p$, the exit probability of
Section~\ref{exit-probabilities} is
$r(a, b, x) = (1 - \rho^{\,x+a+1})/(1 - \rho^{\,a+b+2})$. When $\rho < 1$,
equivalently $p > 1/2$, numerator and denominator are both close to $1$
for moderate arc lengths, so the right-exit probability is close to $1$:
a clockwise-biased walk almost always exits the arc to the right.
Conversely, when $\rho > 1$ the right-exit probability approaches $0$.

\begin{theorem}[Last-Coverage Formula for the Biased Walk] \label{Theorem-4.1}
For a biased random walk on $C_n$ with fixed clockwise probability
$p \in (0,1)$, $p \neq 1/2$, the last-coverage probabilities are
\[
P(L = k) = \frac{(1-\rho)\rho^{\,n-1-k}}{1-\rho^{\,n-1}}
\]
for $k=1,\ldots,n-1$, where $\rho = (1-p)/p = q/p$.
\end{theorem}

\begin{proof}
\textbf{Step 1} (A two-state time-inhomogeneous Markov chain on
expansion directions). Set $\rho = q/p$ and $u_i = 1-\rho^i$ for
$i=0,1,\ldots$. By Theorem~\ref{theorem-3.6} the coverage process is a
sequence of $n-1$ left/right expansions, and the exit probabilities
$r_p(a,b,x)$ are instances of the gambler's ruin formula, so the number
of admissible states at expansion step $i$ depends only on $i$. This
lets us represent the directions of successive expansions
$Y_1, Y_2, \ldots \in \{L,R\}$ as a \emph{two-state time-inhomogeneous
Markov chain}, where the state at step $i$ is the direction $Y_i$ of the most
recent expansion, and the transition probability to the next direction
depends on $i$ (through the current boundary sizes) but, crucially, not
on the history beyond $Y_i$. Explicitly, define
\[
P^i_{x,y} = \Pr(\text{next exit is } y \mid \text{last exit was } x
\text{ and } i \text{ vertices already found})
\]
for $x,y \in \{L,R\}$. Direct application of the biased gambler's ruin
formula gives
\begin{align*}
    P^i_{R,R} &= \frac{1-\rho^{i+1}}{1-\rho^{i+2}} = \frac{u_{i+1}}{u_{i+2}},
    &
    P^i_{R,L} &= 1 - P^i_{R,R} = \frac{\rho^{i+1}u_1}{u_{i+2}}, \\
    P^i_{L,R} &= \frac{1-\rho}{1-\rho^{i+2}} = \frac{u_1}{u_{i+2}},
    &
    P^i_{L,L} &= 1 - P^i_{L,R} = \frac{\rho\, u_{i+1}}{u_{i+2}}.
\end{align*}

\textbf{Step 2} (Recursive formulas for directional sub-counts). Let
$g_R(A,B)$ be the probability that among the first $A+B$ expansions
there are $A$ to the left, $B$ to the right, and the last expansion was
to the right; define $g_L(A,B)$ symmetrically. Then
\[
\Pr(L=k) = g_R(n-k-1, \; k-1) + g_L(n-k-1, \; k-1),
\]
since vertex $k$ is last discovered exactly when, among the preceding
$n-2$ expansions, there are $n-k-1$ left and $k-1$ right expansions (as
in Remark~\ref{remark-Hk}), regardless of the direction of the final
expansion itself.

\textbf{Step 3} (Induction on $m = A+B$). We claim that
\begin{equation} \label{eqn:grgl-formulas}
    g_R(A,B) = \frac{\rho^A u_1 u_B}{u_m u_{m+1}}
    \qquad\text{and}\qquad
    g_L(A,B) = \frac{\rho^m u_1 u_A}{u_m u_{m+1}}.
\end{equation}
For the base case, observe that $u_1^2/(u_1u_2) = p = g_R(0,1)$ and
$\rho u_1^2/(u_1u_2) = q = g_L(1,0)$. For the inductive step, suppose
\eqref{eqn:grgl-formulas} holds for all pairs with sum $m - 1$. Then,
for $m > 1$,
\begin{align*}
    g_R(A,B) &= g_L(A,B-1)\,P^{m-1}_{L,R} + g_R(A,B-1)\,P^{m-1}_{R,R} \\
    &= \frac{\rho^{m-1} u_1 u_A}{u_{m-1} u_{m}} \cdot \frac{u_1}{u_{m+1}}
       + \frac{\rho^A u_1 u_{B-1}}{u_{m-1} u_{m}} \cdot \frac{u_{m}}{u_{m+1}} \\
    &= \frac{u_1}{u_{m-1} u_{m} u_{m+1}}
       \Bigl( \rho^{m-1} u_1 u_A + \rho^A u_{B-1} u_m \Bigr)
    = \frac{u_1}{u_{m-1} u_{m} u_{m+1}}
       \Bigl( \rho^{2m-1} + \rho^{A} - \rho^{m-1+A} - \rho^m \Bigr),
\end{align*}
using the inductive hypothesis. A direct expansion of the claimed closed
form gives the same expression:
\[
    \frac{\rho^A u_1 u_B}{u_m u_{m+1}}
    = \frac{u_1}{u_{m-1} u_{m} u_{m+1}} \bigl( \rho^A u_B u_{m-1} \bigr)
    = \frac{u_1}{u_{m-1} u_{m} u_{m+1}}
      \Bigl( \rho^{2m-1} + \rho^{A} - \rho^{m-1+A} - \rho^m \Bigr).
\]
The inductive step for $g_L$ is analogous: for $m > 1$,
\begin{align*}
    g_L(A,B) &= g_L(A-1,B)\,P^{m-1}_{L,L} + g_R(A-1,B)\,P^{m-1}_{R,L} \\
    &= \frac{\rho^{m-1} u_1 u_{A-1}}{u_{m-1} u_{m}} \cdot \frac{\rho u_{m}}{u_{m+1}}
       + \frac{\rho^{A-1} u_1 u_{B}}{u_{m-1} u_{m}} \cdot \frac{\rho^{m}(1-\rho)}{u_{m+1}} \\
    &= \frac{\rho^m u_1}{u_{m-1} u_{m} u_{m+1}}
       \Bigl( u_{A-1} u_m + \rho^{A-1}u_1u_B \Bigr)
    = \frac{\rho^m u_1}{u_{m-1} u_{m} u_{m+1}}
       \Bigl( \rho^{m+A-1} + 1 - \rho^{m-1} - \rho^{A} \Bigr),
\end{align*}
which again matches the direct expansion
\[
    \frac{\rho^m u_1 u_A}{u_m u_{m+1}}
    = \frac{\rho^m u_1}{u_{m-1} u_{m} u_{m+1}} \bigl( u_A u_{m-1} \bigr)
    = \frac{\rho^m u_1}{u_{m-1} u_{m} u_{m+1}}
      \Bigl( 1 - \rho^{A} - \rho^{m-1} + \rho^{A+m-1} \Bigr).
\]
This establishes \eqref{eqn:grgl-formulas} for all $A,B \geq 0$.

\textbf{Step 4} (Closed form for $P(L=k)$). Combining Steps 2 and 3,
\begin{align*}
\Pr(L=k) &= g_R(n-k-1, \; k-1) + g_L(n-k-1, \; k-1)
= \frac{\rho^{n-k-1} u_1 u_{k-1}}{u_{n-2} u_{n-1}} + \frac{\rho^{n-2} u_1 u_{n-k-1}}{u_{n-2} u_{n-1}} \\
&= \frac{u_1}{u_{n-2}u_{n-1}} \Bigl( \rho^{n-k-1} u_{k-1} + \rho^{n-2} u_{n-k-1} \Bigr)
= \frac{u_1}{u_{n-2}u_{n-1}} \Bigl( \rho^{n-k-1} - \rho^{2n-k-3} \Bigr) \\
&= \frac{u_1}{u_{n-2}u_{n-1}} \Bigl( \rho^{n-k-1} u_{n-2} \Bigr)
= \frac{u_1\rho^{n-k-1}}{u_{n-1}}
=  \frac{(1-\rho)\rho^{\,n-k-1}}{1-\rho^{\,n-1}}. \qed
\end{align*}
\end{proof}

\subsubsection*{Limiting behaviour}\label{limiting-behaviour}

As $p \to 1$ we have $\rho \to 0$, the exit probability
$(1-\rho^x)/(1-\rho^N)$ approaches $1$ for all $x \geq 1$, and nearly
every expansion is clockwise: the arc sweeps clockwise around the ring
and the last vertex incorporated is the one immediately
counterclockwise of the start, so $P(L = n-1) \to 1$. By the analogous
argument in the opposite direction, $P(L = 1) \to 1$ as $p \to 0$. As
$p \to 1/2$ the ratio $\rho \to 1$, the exit probability approaches
$x/N$, and $P(L=k)$ approaches $1/(n-1)$ for every $k$;
Section~\ref{the-unbiased-walk} makes this precise.

\subsubsection*{Interpretation}\label{interpretation}

A fixed directional bias concentrates the last-discovery probability on
the vertex located opposite the dominant direction of expansion. When
the walk is clockwise-biased, the visited arc tends to expand clockwise,
so clockwise vertices are incorporated rapidly. As the arc grows around
the ring, the vertex immediately counterclockwise of the starting point
is typically approached only after the remainder of the cycle has
already been visited, making it the most likely last discovery. When the
walk is counterclockwise-biased, the vertex immediately clockwise of the
start plays the same role. As the bias strengthens, the distribution of
$L$ narrows progressively from the uniform distribution obtained in the
limit $p \to 1/2$ toward a point mass on a single vertex as $p$
approaches $0$ or $1$.

\begin{figure}[ht]
    \centering
    \includegraphics[width=0.75\textwidth]{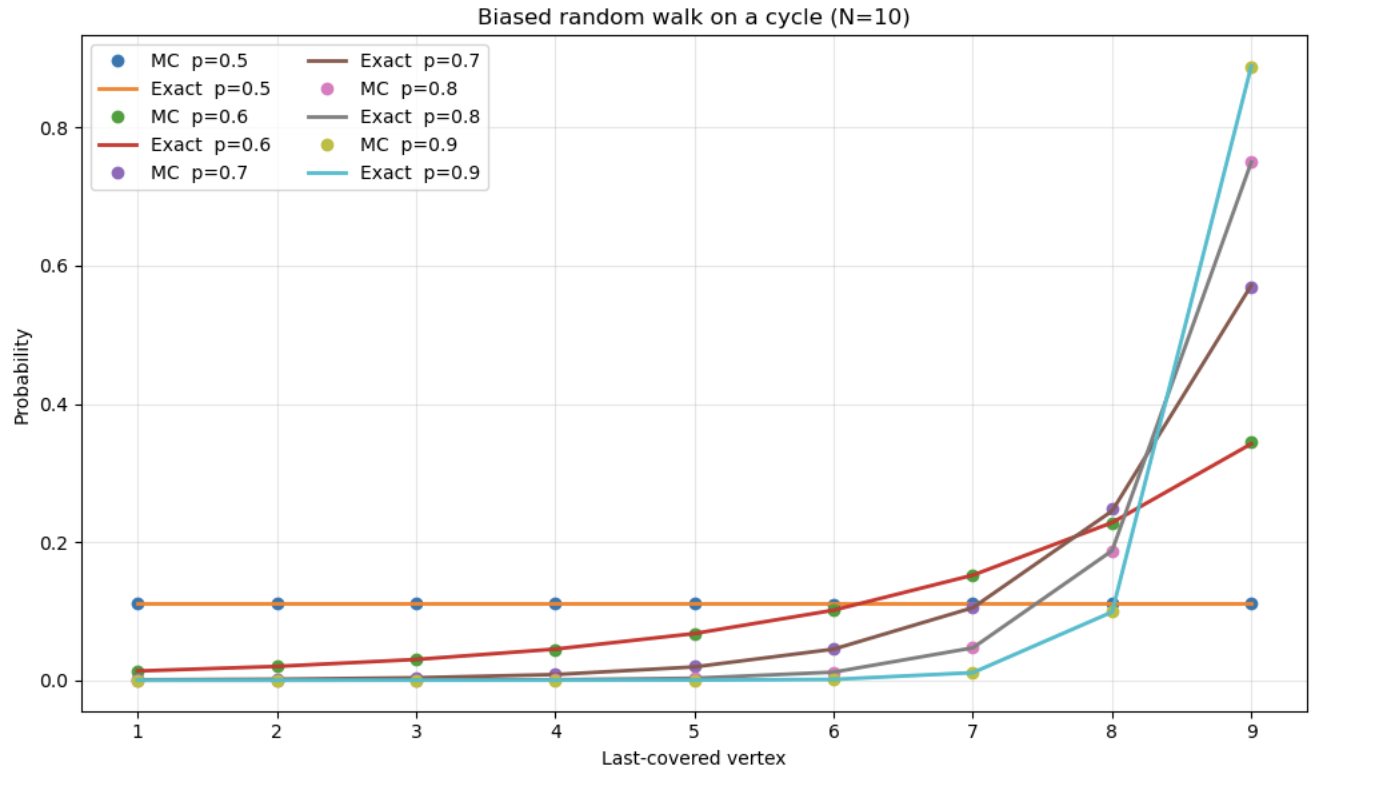}
    \caption{Last-coverage probabilities for a 10-vertex ring under different bias values $p=0.5,0.6,0.7,0.8,$ and $0.9$.}
    \label{fig:biased-walk}
\end{figure}

\subsection{The Unbiased Walk as a Limiting Case}\label{the-unbiased-walk}

Having derived a closed form for the biased walk, we now recover the
classical unbiased result as the limit of that formula as $p \to 1/2$
rather than proving it directly. The formula of
Theorem~\ref{Theorem-4.1} is not itself defined at $p=1/2$, since
$\rho = 1$ makes every factor $1-\rho^i$ vanish. The goal is to
show that the resulting $0/0$ singularity is removable and that the
removed value agrees with the known uniform distribution.

For the unbiased walk the exit probability from an arc with boundaries
$(a,b)$ and the state at position $x$ is $r(a,b,x) = (x+a+1)/(a+b+2)$,
by the symmetric gambler's ruin formula (see Feller, Section XIV.2, and
Levin and Peres, Proposition 9.6)
\cite{feller1971introduction,levin2026markov}. This expression is linear
in $x$, and under the reflection that swaps left and right, that is,
replacing $x$ by $-x$ and interchanging $a$ and $b$, it satisfies
\[
r(b, a, -x) = \frac{-x + b + 1}{b + a + 2} = 1 - r(a, b, x)
\]
\cite{feller1971introduction,levin2026markov}. Reflecting the arc and
the state's position thus exactly interchanges the right- and left-exit
probabilities, indicating that the unbiased walk has no intrinsic
directional preference, consistent with the uniform distribution
established below.

\begin{lemma}[Continuity of $P_p(L=k)$ at $p=1/2$] \label{Lemma-4.2}
Fix a vertex $k \in \{1,\ldots,n-1\}$ and let $F_k(p) = P_p(L=k)$ denote
the last-coverage probability of vertex $k$ as a function of the
clockwise bias $p \in (0,1)$, where $F_k(p)$ is given by
Theorem~\ref{Theorem-4.1} for $p \neq 1/2$. Then $F_k$ \JS{is continuous at $p=1/2$.}
\end{lemma}

\begin{proof}
By Theorem~\ref{theorem-3.6}, $F_k(p) = \sum_{h\in H_k} P_p(h)$, where
$H_k$ is the finite set of expansion sequences in which vertex $k$ is
discovered last and, for a particular sequence $h$,
\[
P_p(h) = \prod_{i=1}^{n-1}
r_p(a_i^h,b_i^h,x_i^h)^{\mathbf{1}\{y_i=R\}}
\bigl(1-r_p(a_i^h,b_i^h,x_i^h)\bigr)^{\mathbf{1}\{y_i=L\}},
\qquad
r_p(a,b,x) = \frac{1-\rho^{\,x+a+1}}{1-\rho^{\,a+b+2}}, \quad \rho = \frac{1-p}{p}.
\]

%\textbf{Step 1} (Continuity for $p \neq 1/2$). The map
%$p \mapsto \rho = (1-p)/p$ is continuous on $(0,1)$. For $p \neq 1/2$ we
%have $\rho \neq 1$, so $1 - \rho^{a+b+2} \neq 0$ and $r_p(a,b,x)$ is a
%continuous function of $p$ on $(0,1)\setminus\{1/2\}$; consequently so
%is $1 - r_p(a,b,x)$. Since $P_p(h)$ is a finite product of such terms
%and $H_k$ is finite, $F_k$ is a finite sum of continuous functions and
%is therefore continuous on $(0,1)\setminus\{1/2\}$.

\textbf{Step 1} (Limit of a single exit probability). Fix $a,b,x$ and
set $m = x+a+1$, $N = a+b+2$, so that
$r_p(a,b,x) = (1-\rho^m)/(1-\rho^N)$. As $p \to 1/2$ we have
$\rho \to 1$ and both numerator and denominator tend to $0$. Applying
L'H\^opital's rule in $\rho$,
\[
\lim_{\rho \to 1} \frac{1-\rho^m}{1-\rho^N}
= \lim_{\rho \to 1} \frac{-m\rho^{m-1}}{-N\rho^{N-1}}
= \frac{m}{N}
= \frac{x+a+1}{a+b+2},
\]
which is exactly the unbiased exit probability above. Hence
$\lim_{p\to 1/2} r_p(a,b,x) = r_{1/2}(a,b,x)$, and by the same argument
$\lim_{p\to1/2}\bigl(1-r_p(a,b,x)\bigr) = 1-r_{1/2}(a,b,x)$.

\textbf{Step 3} (Limit of $F_k$). Since $\lim_{p\to1/2} r_p(a,b,x)$ exists and is finite, the limit can be pulled outside the product in the expression for $P_p(h)$. This implies that $P_p(h)$ is continuous at $p=1/2$:
\begin{align*}
P_{1/2}(h) &= \prod_{i=1}^{n-1} \lim_{p \to 1/2} r_p(a_i^h,b_i^h,x_i^h)^{\mathbf{1}\{y_i=R\}}
\bigl(1-r_p(a_i^h,b_i^h,x_i^h)\bigr)^{\mathbf{1}\{y_i=L\}} \\
&= \lim_{p \to 1/2} \prod_{i=1}^{n-1} r_p(a_i^h,b_i^h,x_i^h)^{\mathbf{1}\{y_i=R\}}
\bigl(1-r_p(a_i^h,b_i^h,x_i^h)\bigr)^{\mathbf{1}\{y_i=L\}} \\
&= \lim_{p \to 1/2} P_p(h).
\end{align*}

Hence,
\begin{align*}
    F_k(1/2) &= \sum_{h \in H_k} P_{1/2}(h) \\
    &= \sum_{h \in H_k} \lim_{p \to 1/2} P_{p}(h) \\
    &= \lim_{p \to 1/2} \sum_{h \in H_k} P_{p}(h) \\
    &= \lim_{p \to 1/2} F_k(p).
\end{align*}

% Since each $h \in H_k$ involves only
% finitely many exit-probability factors, Step~2 gives
% $\lim_{p\to1/2^-} P_p(h) = \lim_{p\to1/2^+} P_p(h) = P_{1/2}(h)$ for
% every $h \in H_k$, where $P_{1/2}(h)$ is the corresponding product of
% unbiased exit probabilities. Because $H_k$ is finite, the limit may be
% taken term by term:
% \[
% \lim_{p\to1/2^\pm} F_k(p)
% = \sum_{h\in H_k} \lim_{p\to1/2^\pm} P_p(h)
% = \sum_{h\in H_k} P_{1/2}(h)
% = P_{1/2}(L=k).
% \]
% Combined with Step~1, this shows the $0/0$ singularity of $F_k$ at
% $p=1/2$ is removable: setting $F_k(1/2) := P_{1/2}(L=k)$ makes $F_k$
% continuous on all of $(0,1)$. 
\end{proof}

% \textbf{Step 4} (Evaluating the limiting value). It remains to identify
% $P_{1/2}(L=k)$. For the unbiased walk the last vertex is uniformly
% distributed among the $n-1$ non-starting vertices, so
% $P_{1/2}(L=k) = 1/(n-1)$; this is the classical result of
% \cite{aleliunas1979random,levin2026markov,lovasz1993note,buchanan2022last},
% and is re-derived independently as Corollary~\ref{Corollary-4.3} below
% directly from Theorem~\ref{Theorem-4.1} together with Steps~1--3, without
% appealing to the cited literature. Either way,
% $\lim_{p\to1/2} F_k(p) = 1/(n-1)$.
% \phantom{}

\begin{corollary}[Uniform Last-Coverage Distribution] \label{Corollary-4.3}
For an unbiased random walk on $C_n$ starting at vertex $0$ (that is,
$p = 1/2$), every non-starting vertex is equally likely to be discovered
last: $P(L = k) = 1/(n - 1)$ for every $k \in \{1, 2, \ldots, n - 1\}$.
\end{corollary}

\begin{proof}
By Theorem~\ref{Theorem-4.1}, for every $p \in (0,1)\setminus\{1/2\}$ we
have $P_p(L=k) = (1-\rho)\rho^{\,n-1-k}/(1-\rho^{\,n-1})$ with
$\rho = (1-p)/p$, and by Lemma~\ref{Lemma-4.2} the map
$p \mapsto P_p(L=k)$ extends continuously to $p=1/2$ with
\[
P_{1/2}(L=k) = \lim_{p\to1/2} P_p(L=k) = \lim_{\rho\to1} \frac{(1-\rho)\rho^{\,n-1-k}}{1-\rho^{\,n-1}}.
\]
As $\rho \to 1$, both numerator and denominator tend to $0$. Writing
$j = n-1-k$ and applying L'H\^opital's rule in $\rho$,
\[
\frac{d}{d\rho}\Bigl[(1-\rho)\rho^{j}\Bigr]
= -\rho^{j} + (1-\rho)\,j\,\rho^{j-1}
\;\xrightarrow[\rho\to1]{}\; -1,
\qquad
\frac{d}{d\rho}\Bigl[1-\rho^{n-1}\Bigr]
= -(n-1)\rho^{n-2}
\;\xrightarrow[\rho\to1]{}\; -(n-1),
\]
so $P_{1/2}(L=k) = (-1)/\bigl(-(n-1)\bigr) = 1/(n-1)$. This value is the
same for every $k \in \{1,\ldots,n-1\}$.
\end{proof}

\begin{remark}
Corollary~\ref{Corollary-4.3} is derived purely from
Theorem~\ref{Theorem-4.1} and Lemma~\ref{Lemma-4.2}, without
re-deriving the uniform distribution from Theorem~\ref{theorem-3.6} by an
independent combinatorial argument, demonstrating that the closed-form
biased formula is consistent with, and in fact implies, the classical
unbiased result. Both possible directions of the final expansion are
already accounted for, via the combination $g_R + g_L$ in Step~2 of the
proof of Theorem~\ref{Theorem-4.1}, so no separate case analysis is
needed.
\end{remark}

Corollary~\ref{Corollary-4.3} appears counterintuitive at first. One
expects vertices far from the starting point, such as the antipodal
vertex on a large ring, to be harder to reach and hence more likely to
be discovered last. The expectation is flawed because the
interval-expansion process is perfectly balanced under the unbiased
walk. The arc is equally likely to extend left or right at each
juncture, and no vertex enjoys any structural advantage from the
perspective of the growing arc boundaries, so graph distance from the
start does not translate into any difference in last-coverage
probability. As soon as any bias is introduced, however small, the
distribution ceases to be uniform and concentrates away from the
dominant direction of travel; Corollary~\ref{Corollary-4.3} is exactly
the statement that this concentration vanishes continuously as the bias
is turned off.

\subsubsection*{Numerical verification}\label{numerical-verification}

Corollary~\ref{Corollary-4.3} was verified by Monte Carlo simulation.
For ring sizes $n = 5, 10, 20, 50,$ and $100$, twenty thousand
independent walks were simulated, each continuing until every vertex had
been visited, and the fraction of simulations in which each vertex was
the last discovery was recorded.

In each case the observed fractions were close to the theoretical value
$1/(n-1)$, with deviations consistent with the expected statistical
error from the sample size. For $n = 10$, for which the theoretical
probability is $1/9 \approx 0.1111$, the observed fractions ranged from
$0.1104$ to $0.1121$ across the nine non-starting vertices.

\begin{figure}[ht]
    \centering
    \includegraphics[width=0.7\textwidth]{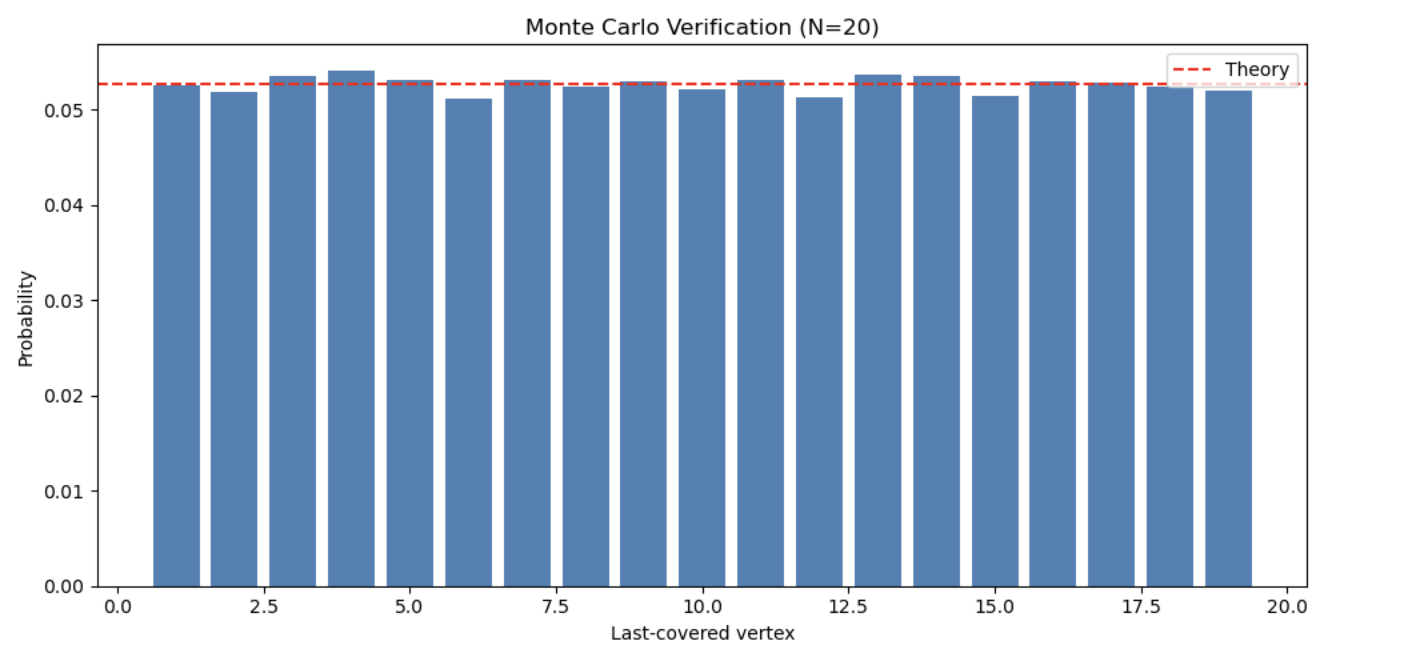}
    \caption{Observed last-coverage frequencies for the unbiased walk, compared with the theoretical value.}
    \label{fig:unbiased-simulation}
\end{figure}

\subsection{The Time-Varying Bias Walk}\label{the-time-varying-bias-walk}

The most general model allows the directional preference to change at
each step: the clockwise probability at step $t$ is $p(t)$, taken from a
deterministic schedule that is known in advance but otherwise arbitrary,
with all values strictly between $0$ and $1$ so that both directions
remain possible. This encompasses many practical scenarios, including a
patrol robot that alternates direction on a fixed schedule, a routing
protocol that adjusts forwarding probabilities by time of day, or a walk
modulated by an external periodic signal.

To determine the next discovery at any point, three variables suffice:
the boundary sizes $A_t$ and $B_t$ and the state's current position
$X_t$ within the arc. Together with the time $t$, which determines
$p(t)$, these constitute the full state $(A_t, B_t, X_t, t)$; for a
given ring size the number of possible triples $(a,b,x)$ is finite and
bounded by $O(n^3)$.

\begin{corollary}[Exact Last-Coverage Formula, Time-Varying Walk] \label{corollary-4.5}
For any deterministic schedule $p(t)$ satisfying
$c \leq p(t) \leq 1-c$ for all $t \geq 0$ and some constant $c > 0$, the
last-coverage probability is
\[
P(L = k)
= \sum_{t=n-1}^{\infty}\; \sum_{h \in H_k^t} P(h),
\qquad
P(h)
= \prod_{s=1}^{t}
  p(s)^{\mathbf{1}\{I_s=+1\}}\,(1-p(s))^{\mathbf{1}\{I_s=-1\}},
\]
where $H_k^t$ denotes the set of $t$-step sequences
$(I_1,\ldots,I_t)\in\{+1,-1\}^t$ that achieve complete coverage of $C_n$
for the first time at step $t$ with vertex $k$ as the final discovery.
This series is generally infinite; Algorithm~1 below gives a computable,
finitely truncated approximation to it with an explicit error bound.
\end{corollary}

\begin{proof}
\textbf{Step 1} (Indicator-sum representation of coverage). Represent
each step $s$ by a signed indicator $I_s \in \{+1,-1\}$, with
$P(I_s = +1) = p(s)$ (clockwise) and $P(I_s = -1) = 1-p(s)$
(counterclockwise). Define $S_t = \sum_{s=1}^{t} I_s$, so the walker's
displacement from vertex $0$ after $t$ steps is $S_t \pmod{n}$. By the
Connected Arc Property, complete coverage first occurs at
\[
T_{\mathrm{cover}}
= \min\!\Bigl\{t \geq 0 :
  \max_{0\leq s\leq t} S_s - \min_{0\leq s\leq t} S_s \geq n-1
  \Bigr\}.
\]
Partitioning over the realised value of $T_{\mathrm{cover}}$ and over
the sequences of that length that result in vertex $k$ being the final
discovery gives the stated double sum.

\textbf{Step 2} (Convergence). The series converges because
$P(T_{\mathrm{cover}} < \infty) = 1$. Since $p(s) \geq c > 0$ for all
$s$, in any block of $n$ consecutive steps the probability that all
steps are clockwise is at least $c^n > 0$. By a geometric-trials
argument, the range $R_t = \max_{s\leq t} S_s - \min_{s\leq t} S_s$
reaches $n-1$ in finitely many such blocks with probability $1$;
specifically,
\[
P(T_{\mathrm{cover}} > kn)
\;\leq\;
\bigl(1-c^n\bigr)^{\lfloor k/(n-1)\rfloor},
\]
which decays geometrically in $k$
\cite{levin2026markov,aldous2002reversible}. Therefore
$\mathbb{E}[T_{\mathrm{cover}}] < \infty$,
$P(T_{\mathrm{cover}} < \infty) = 1$, and $\sum_{k=1}^{n-1} P(L=k) = 1$,
confirming the series is properly normalised.
\end{proof}

\medskip
\noindent\rule{\linewidth}{0.4pt}
\nopagebreak
\medskip

\noindent\textbf{Algorithm 1 (Truncated Series Approximation).}
The exact series in Corollary~\ref{corollary-4.5} sums over infinitely
many step sequences and is not directly computable. A practical
approximation fixes a finite truncation horizon $T \geq n-1$ and sums
only over sequences of length at most $T$:
\[
\widehat{P}_T(L = k)
= \sum_{t=n-1}^{T}\;\sum_{h \in H_k^t} P(h).
\]

\begin{enumerate}
\item \textbf{Fix a horizon.} Choose $T \geq n-1$. For each length $t$
      from $n-1$ to $T$, consider all $2^t$ step sequences
      $h = (I_1,\ldots,I_t)\in\{+1,-1\}^t$.

\item \textbf{Check first-coverage at step $t$.} Compute the partial
      sums $S_s = \sum_{j=1}^{s} I_j$ for $s=0,\ldots,t$. The sequence
      achieves complete coverage for the \emph{first} time at step $t$
      if and only if
      \[
        \max_{0\leq s\leq t} S_s - \min_{0\leq s\leq t} S_s = n-1
        \quad\text{and}\quad
        \max_{0\leq s\leq t-1} S_s - \min_{0\leq s\leq t-1} S_s < n-1 .
      \]
      Sequences failing either condition are discarded.

\item \textbf{Identify the last-covered vertex.} For a surviving
      sequence, the final step extends one boundary of the arc for the
      last time. The last-covered vertex is
      $\bigl(\max_{s\leq t} S_s + 1\bigr)\bmod n$ if the final step
      extended the right boundary ($I_t = +1$ and $S_t > S_{t-1}$), and
      $\bigl(\min_{s\leq t} S_s - 1\bigr)\bmod n$ if it extended the
      left boundary. If this vertex is not $k$, discard the sequence.

\item \textbf{Accumulate probability.} For each surviving
      $h \in H_k^t$, compute $P(h)$ as in
      Corollary~\ref{corollary-4.5} and add it to the running total
      $\widehat{P}_T(L=k)$.
\end{enumerate}

\noindent\textbf{Convergence and error bound.} As $T\to\infty$,
$\widehat{P}_T(L=k)\to P(L=k)$ monotonically from below, with
\[
0
\;\leq\;
P(L=k) - \widehat{P}_T(L=k)
\;\leq\;
P(T_{\mathrm{cover}} > T)
\;\leq\;
\bigl(1-c^n\bigr)^{\lfloor T/(n(n-1))\rfloor},
\]
so choosing $T$ with $(1-c^n)^{\lfloor T/(n(n-1))\rfloor} < \varepsilon$
guarantees error at most $\varepsilon$. Since $\widehat{P}_T(L=k)$ is
the only quantity from this subsection that is evaluated numerically,
the time-varying results reported below are truncated, error-bounded
approximations rather than exact values. The truncated sum enumerates at
most $\sum_{t=n-1}^{T} 2^t = O(2^T)$ sequences, which is exponential in
$T$ and therefore practical only for moderate $T$ and $n$; reducing this
cost, for instance by evaluating the recursion directly over the
$O(n^3)$ states $(a,b,x)$ rather than by exhaustive enumeration, is a
natural improvement but is not required for the correctness of the
guarantee above.

\medskip
\noindent\rule{\linewidth}{0.4pt}
\medskip

Two schedules illustrate the range of behaviour. Under a sinusoidal bias
$p(t) = 1/2 + \epsilon\sin(\omega t)$, the walk alternates between
clockwise and counterclockwise preference, and a small $\epsilon$ makes
the last-coverage distribution a small perturbation of the uniform one.
A high oscillation frequency causes the two biases to average out over
successive steps, so the distribution approaches the uniform
distribution of Corollary~\ref{Corollary-4.3}, whereas a low frequency
lets the bias persist for many steps before reversing, producing more
pronounced departures from uniformity; Section~\ref{monte-carlo-simulation}
examines this quantitatively. Under a piecewise constant bias, with
$p(t) = p_1$ before time $\tau$ and $p(t) = p_2$ thereafter, the early
phase guides which vertices are discovered first while the late phase,
during which the final few vertices are reached, dominates the identity
of the last. Algorithm~1 accommodates either schedule without
modification.

\section{Monte Carlo Simulation}\label{monte-carlo-simulation}

This section presents simulations that complement the theoretical
results: first the effect of a sinusoidally time-varying bias, then the
relationship between graph symmetry and last-coverage uniformity.

\subsection{Variance Analysis of the Sinusoidal Walk}\label{variance-analysis-of-the-sinusoidal-time-varying-walk}

For a ring of size $n = 10$, twenty thousand independent walks were
simulated for each combination of
$\epsilon \in \{0, 0.05, 0.1, 0.2, 0.3\}$ and
$\omega \in \{0.25, 0.5, 1, 2, 3, 4, 5, 6\}$ under the schedule
$p(t) = 1/2 + \epsilon\sin(\omega t)$, with $p(t)$ clipped to
$[0.001, 0.999]$ for numerical stability. For each combination the
empirical distribution of $L$ was used to estimate its mean and
variance, treating the vertex labels $1, \ldots, n-1$ as numerical
values. Table~\ref{tab:variance} reports the estimated variances.

\begin{table}[ht]
\centering
\caption{Estimated variance of the last-coverage vertex $L$ on $C_{10}$
under the sinusoidal schedule $p(t) = 1/2 + \epsilon\sin(\omega t)$,
from twenty thousand Monte Carlo trials per cell.}
\label{tab:variance}
\begin{tabular}{c|cccccccc}
\toprule
$\epsilon$ & \multicolumn{8}{c}{$\omega$} \\
\cmidrule(l){2-9}
 & 0.25 & 0.5 & 1.0 & 2.0 & 3.0 & 4.0 & 5.0 & 6.0 \\
\midrule
0.00 & 6.657 & 6.642 & 6.707 & 6.741 & 6.649 & 6.607 & 6.737 & 6.737 \\
0.05 & 6.556 & 6.723 & 6.624 & 6.667 & 6.675 & 6.637 & 6.647 & 6.600 \\
0.10 & 6.351 & 6.541 & 6.653 & 6.685 & 6.692 & 6.694 & 6.657 & 6.358 \\
0.20 & 5.497 & 6.182 & 6.365 & 6.592 & 6.738 & 6.613 & 6.418 & 5.697 \\
0.30 & 4.269 & 5.551 & 6.062 & 6.527 & 6.744 & 6.544 & 6.259 & 4.501 \\
\bottomrule
\end{tabular}
\end{table}

Three trends are apparent. First, at $\epsilon = 0$ the schedule reduces
to the unbiased walk regardless of $\omega$, and the estimated variance
is close to the exact value $\mathrm{Var}(L) = (n^2-1)/12 = 6.667$
implied by Corollary~\ref{Corollary-4.3} at every $\omega$ tested, up to
Monte Carlo error; this is a sanity check on the simulation before
$\epsilon$ is varied.

Second, for fixed $\omega$, increasing $\epsilon$ generally pushes the
variance further from that baseline, and the size of the shift is not
uniform across $\omega$: at $\omega = 3$ the variance stays within
$6.65$--$6.74$ across every $\epsilon$ tested, while at $\omega = 0.25$
it falls monotonically from $6.657$ at $\epsilon = 0$ to $4.269$ at
$\epsilon = 0.3$. This is consistent with
Section~\ref{the-biased-walk}: a sustained directional preference
concentrates the last-coverage distribution, and larger $\epsilon$
strengthens that preference whenever the schedule fails to average out
over a typical walk.

Third, for fixed $\epsilon > 0$ the variance is not monotonic in
$\omega$. It is depressed at the smallest frequency tested, rises toward
the baseline as $\omega$ increases through $2$--$4$, and falls again at
$\omega = 5$ and $\omega = 6$. The typical cover time of a ten-vertex
ring is of order $n^2 \approx 100$ steps, so a schedule with
$\omega = 0.25$ (period $\approx 25$ steps) completes only a fraction of
a cycle over a typical walk; the walk experiences $p(t)$ as an almost
constant drift, reproducing the concentration effect of the fixed-bias
walk. Intermediate frequencies complete several oscillations within a
typical cover time, so the two phases largely cancel and the
distribution of $L$ stays close to uniform.

The renewed dip near $\omega = 6$ is a discrete-time aliasing effect
rather than a second high-frequency regime. Since $t$ is restricted to
integers, $\sin(\omega t) = \sin((\omega - 2\pi)t)$ for every integer
$t$, so a schedule with $\omega$ close to $2\pi \approx 6.283$ is
observationally close to one with a small \emph{negative} effective
frequency $\omega - 2\pi \approx -0.283$, comparable in magnitude to
$\omega = 0.25$. This is corroborated by the estimated mean of $L$,
recorded alongside the variance: at $\epsilon = 0.3$ the mean is pulled
to approximately $6.9$ at $\omega = 0.25$ and to approximately $3.3$ at
$\omega = 6$, an approximately symmetric displacement in opposite
directions from the uniform mean of $5$. The two low-variance regions
therefore reflect the same mechanism, a persistent near-constant bias,
acting in opposite directions. The governing quantity is thus not
$\omega$ in isolation but the number of oscillations completed within a
typical cover time, together with the aliased frequency
$\min(\omega, \,2\pi - \omega)$ induced by integer-step sampling.

\begin{figure}[ht]
    \centering
    \includegraphics[width=0.49\linewidth]{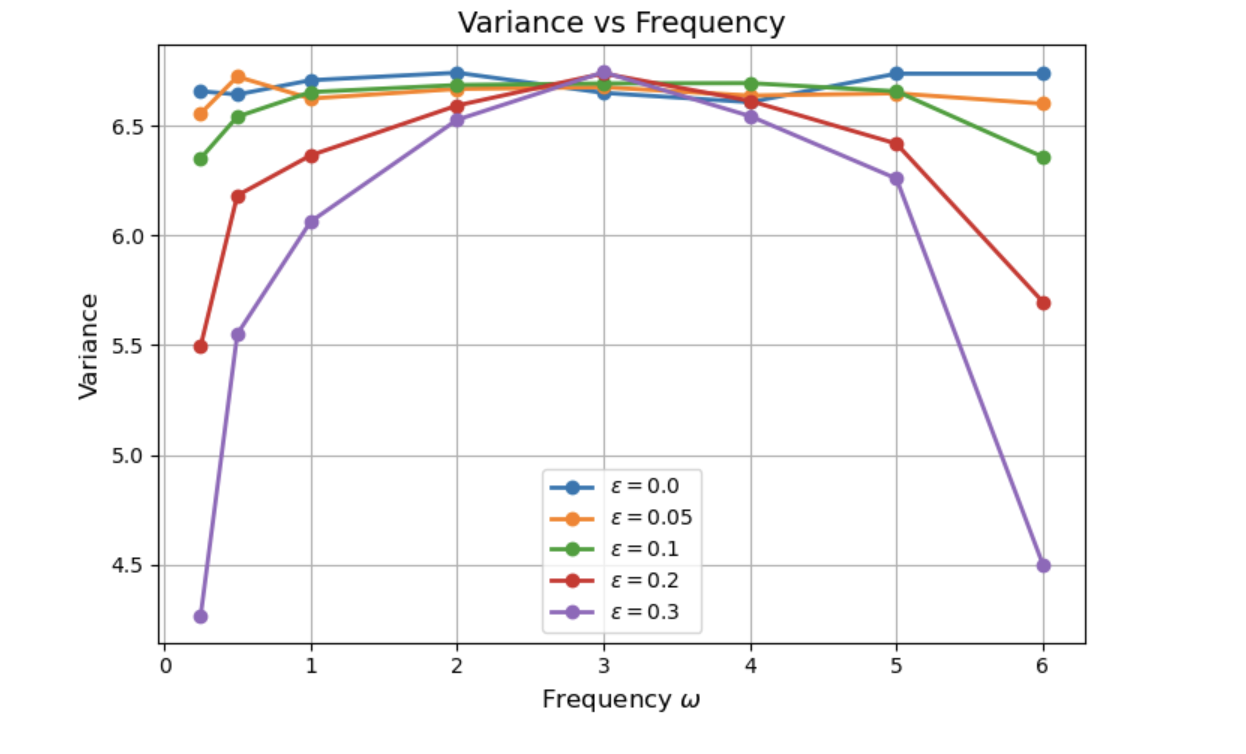}%
    \includegraphics[width=0.49\linewidth]{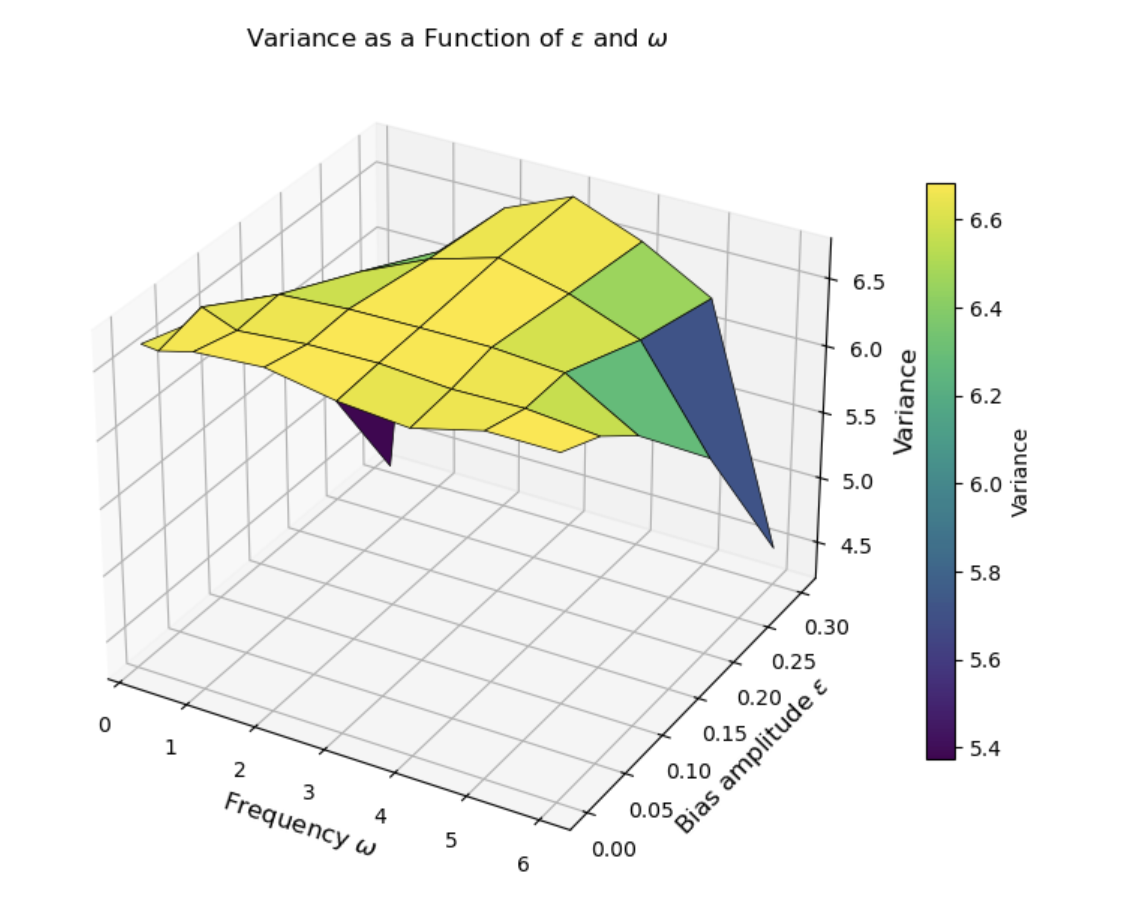}
    \caption{(left) Variance of the last-coverage vertex versus oscillation frequency $\omega$ for different bias amplitudes $\epsilon$. (right) The same data as a three-dimensional surface over $(\epsilon,\omega)$. Both are estimated from 20,000 Monte Carlo simulations on $C_{10}$.}
    \label{fig:variance}
\end{figure}

\subsection{Symmetry and Last-Coverage Probabilities}\label{symmetry-and-last-coverage}

Last-coverage probabilities are uniform for the unbiased random walk on
cycles, which is perhaps counterintuitive, as one might expect more
distant nodes to be more likely to be covered last. They are also
uniform for unbiased walks on complete graphs, and these are the only
examples \cite{lovasz1993note}. This raises the question of whether
there is a connection between symmetry in graphs and uniformity of
last-coverage probabilities.

To gain some intuition, we consider a different class of highly
symmetric graphs: the subset of generalized Petersen graphs $\GP(n,k)$
which are vertex-transitive. These are composed of two cycles of length
$n$ connected by a matching, with the $i$th node of the first cycle
adjacent to node $ik \pmod{n}$ in the second. For example, $\GP(8,3)$ is
shown below. For all of these graphs we take the transition
probabilities to be uniform across the neighbours of the current node.

\begin{center}
\begin{tikzpicture}[scale=0.85]

\tikzstyle{vertex}=[draw, circle, color = black, fill=black, text opacity = 1, inner sep = 1.4pt]

\foreach \s in {0,...,7}{
    \node[vertex] (v\s) at ({360/8 * (-\s)}:1.8cm){};
    \node[vertex] (u\s) at ({360/8 * (-\s)}:1.1cm){};
}
\foreach \s [evaluate = \r using {int(mod(\s+1,8))}] in {0,...,7}{
    \draw (v\s) -- (v\r);
}
\foreach \s [evaluate = \r using {int(mod(\s+3,8))}] in {0,...,7}{
    \draw (u\s) -- (u\r);
    \draw (v\s) -- (u\s);
}

\end{tikzpicture}
\end{center}

Since $\GP(n,k)$ is neither a cycle nor a complete graph, its
last-coverage probabilities are not uniform. To gauge how non-uniform
they are, we compare them with those of Erd\H{o}s--R\'enyi random graphs
$G(n,1/2)$, which include each edge independently with probability
$1/2$. Figure~\ref{fig:symmetry-uniformity} compares all
vertex-transitive generalized Petersen graphs with at most 22 vertices
against Erd\H{o}s--R\'enyi graphs of the same sizes. The standard
deviation of last-coverage probabilities in the generalized Petersen
graphs is markedly lower, supporting the idea that structural symmetry
can produce more uniform last-coverage probabilities.

\begin{figure}[ht]
    \centering
    \includegraphics[width=0.5\linewidth]{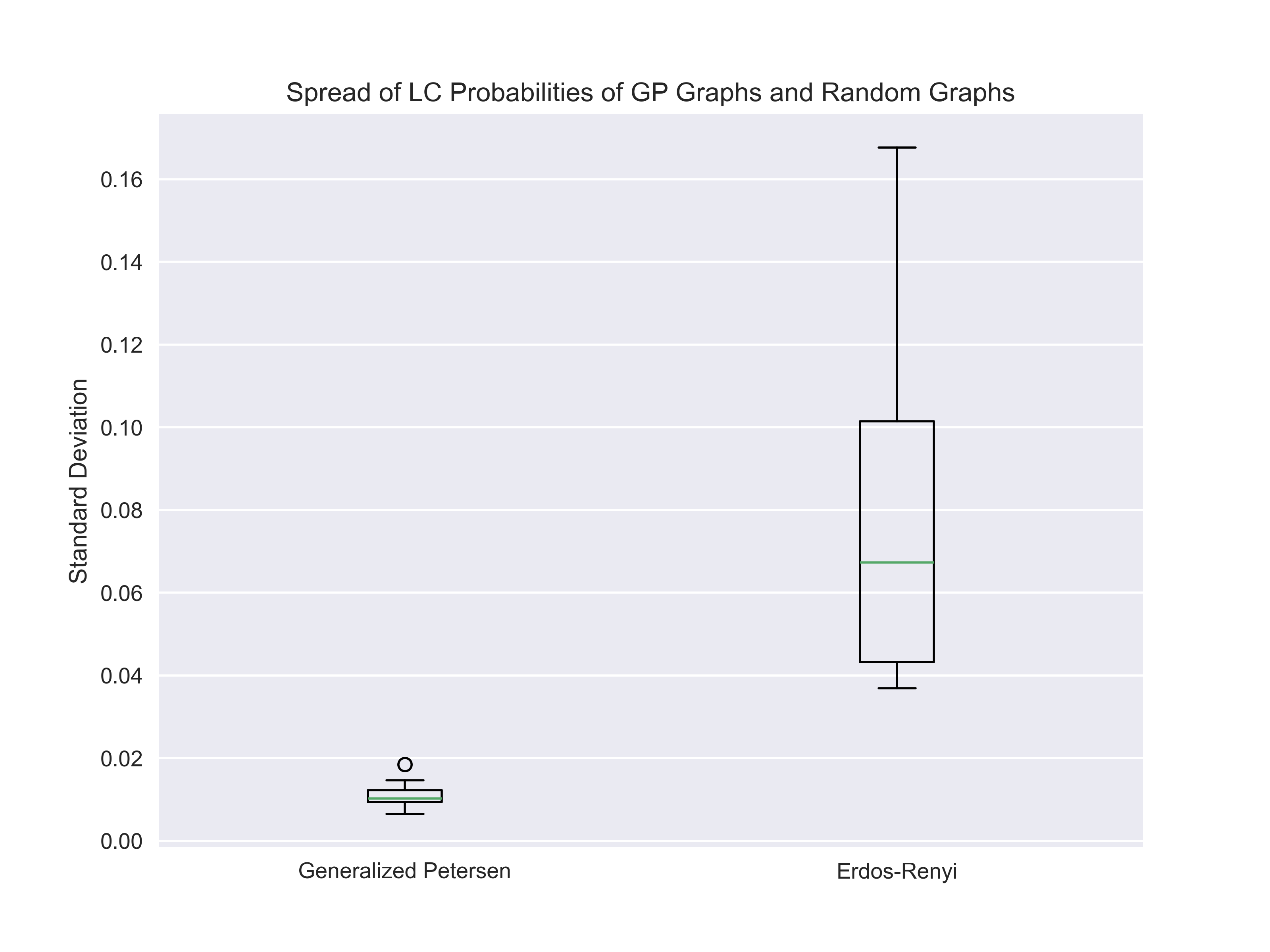}
    \caption{Standard deviation of the last-coverage probabilities of the non-source nodes in two graph families. (left) All vertex-transitive generalized Petersen graphs with at most 22 nodes. (right) Ten $G(n,1/2)$ graphs each of sizes $n=6, 8, \ldots, 22$, matching the sizes of the generalized Petersen graphs.}
    \label{fig:symmetry-uniformity}
\end{figure}

\section{Applications, Open Problems, and Conclusion}\label{applications-open-problems-and-conclusion}

\subsection{Applications}\label{applications}

In ring-topology computer networks, nodes are arranged in a cycle and
communicate by passing control messages, called tokens, around the ring;
the IEEE 802.5 Token Ring standard is a well-established protocol of
this type. When token forwarding follows a random-walk-like rule, the
last-coverage vertex corresponds to the last node to receive an initial
token \cite{donnan2005token}, and random walks are a space-efficient
mechanism for network exploration with cover time of order $n^2$ on a
ring \cite{aleliunas1979random,levin2026markov}. For rings exhibiting
directional preference, as in certain directional routing protocols,
Theorem~\ref{Theorem-4.1} quantifies how the bias shifts the
last-coverage distribution and identifies which nodes are most at risk
of being last, while Corollary~\ref{Corollary-4.3} provides the
complementary fairness guarantee for symmetric protocols: every node is
equally likely to be last, with none at a structural disadvantage. The
same reading applies to epidemic-style information propagation on a ring
\cite{boyd2006randomized}, where the last-informed node is the
last-coverage vertex, so Theorem~\ref{Theorem-4.1} quantifies the
asymmetry in information latency under directed forwarding.

A mobile robot inspecting every point along a circular track faces a
coverage problem modelled directly by a random walk on $C_n$; robotic
coverage strategies are surveyed by Choset and by Galceran and Carreras
\cite{choset2001coverage,galceran2013survey,choset1998coverage}. The
last-coverage vertex is the location most likely to be inspected last,
and hence the weakest point in the patrol pattern.
Theorem~\ref{Theorem-4.1} identifies it for a patrol with a consistent
directional tendency, Corollary~\ref{Corollary-4.3} guarantees that a
symmetric patrol overlooks no location systematically, and
Corollary~\ref{corollary-4.5} enables schedule optimisation, since for
any deterministic $p(t)$ the last-coverage probabilities can be
approximated to any desired accuracy, allowing a designer to search for
a schedule that minimises expected cover time or equalises last-coverage
probabilities.

The framework applies equally to search-and-detection problems with
stationary targets. Consider a maritime search over a circular region
discretised into $n$ sectors, with a helicopter moving between adjacent
sectors according to a schedule $p(t)$ and a stationary target, such as
wreckage or a distress signal, in one of the $n-1$ non-starting sectors.
The last-coverage vertex identifies the sector most likely to remain
unsearched, and if the search has proceeded without locating the target
the conditional probability that it lies in sector $k$ is proportional
to $P(L=k)$. For a search biased by wind or current,
Theorem~\ref{Theorem-4.1} gives the exact skewed distribution,
identifying the sector just behind the dominant search direction as the
most likely remaining location; for a symmetric search no sector is more
suspect than any other. Search coordinators can use these probabilities
to allocate additional resources to the most likely unsearched sector.

From a theoretical perspective, the interval-expansion representation
and its two-state Markov chain refinement show that the last-coverage
problem on a ring can be solved from a reduced description of the
process: it suffices to monitor the two boundary sizes, or equivalently
only the direction of the most recent expansion. The exit probabilities
$r(a,b,x)$ play the role of harmonic functions for the walk restricted
to the current arc, and the associated value function solves a discrete
boundary value problem on the state space
\cite{levin2026markov,doob1984classical}.

\subsection{Open Problems and Future Directions}\label{open-problems-and-future-directions}

\begin{itemize}
\item
  \emph{Higher-dimensional periodic lattices.} The cycle $C_n$ is a
  one-dimensional torus. The $d$-dimensional torus does not share the
  Connected Arc Property, since the visited region need not be
  contiguous, so extending the analysis would require substantially new
  methods. Cover times there are known, of order $n^2$ for $d=1$,
  $n(\log n)^2$ for $d=2$, and $n\log n$ for $d \geq 3$, but the
  distribution of the last-coverage vertex is open for all $d \geq 2$
  \cite{levin2026markov,aldous2002reversible}.
\item
  \emph{Multiple walkers.} With $k$ independent walkers from different
  starting points, the covered set is the union of the individual arcs,
  and the ring is covered roughly $\log k$ times faster in expectation
  \cite{alon2008many}. The $k$ walks form a joint Markov chain on
  $C_n^k$, so the last-coverage distribution can in principle be
  computed via absorption probabilities of that chain, but the state
  space grows exponentially in $k$ and an efficient reduced description
  analogous to the interval-expansion framework remains open.
\item
  \emph{Random or position-dependent step probabilities.} If the
  clockwise probability at each site is drawn at random, the problem
  becomes a random walk in a random environment, surveyed by Zeitouni
  \cite{zeitouni2004part}. If instead it depends deterministically on
  position rather than time, gambler's ruin no longer applies directly,
  though the problem reduces to absorption probabilities of an
  inhomogeneous chain on the cycle, solvable by standard
  linear-algebraic methods.
\item
  \emph{Continuous-time walks.} Replacing the discrete-time walk by a
  continuous-time chain with exponential holding times leaves the
  embedded sequence of visited vertices, and hence the distribution of
  the last-coverage \emph{vertex}, unchanged. What remains open is the
  distribution of the coverage \emph{time}, which does depend on the
  holding-time distribution.
\item
  \emph{Optimal patrol schedule design.} Corollary~\ref{corollary-4.5}
  approximates $P(L=k)$ for any deterministic schedule. Finding the
  schedule that minimises expected cover time, or that minimises
  $\max_k P(L=k)$ so as to make coverage as uniform as possible, is a
  non-trivial optimisation over the space of step-probability sequences.
\item
  \emph{Concentration of the last-coverage vertex.} As the ring grows
  large, one may ask whether $L$ concentrates at a particular location
  and how large $P(L=k)$ can be for the most likely vertex. Answering
  this would require large deviation estimates and would connect the
  present work to cover times and last-passage percolation.
\end{itemize}

\subsection{Conclusion}\label{conclusion}

This paper examines the distribution of the last-coverage vertex in a
nearest-neighbour random walk on the cycle graph $C_n$. Its central
analytical contribution is the interval-expansion framework, built on
Theorem~\ref{Theorem-3.1}, which shows that the visited region always
forms a contiguous arc and can therefore be tracked by two boundary
counters rather than by the full trajectory, and
Theorem~\ref{theorem-3.6}, which shows that the last-coverage
probabilities can be computed exactly by summing products of gambler's
ruin exit probabilities over the relevant expansion sequences.

Three quantitative results follow. Theorem~\ref{Theorem-4.1} gives an
explicit closed form for the biased walk, obtained by representing the
boundary expansions as a two-state time-inhomogeneous Markov chain, and
shows that a fixed directional preference concentrates the distribution
on the vertex just beyond the dominant direction of travel.
Corollary~\ref{Corollary-4.3} recovers the uniform distribution
$1/(n-1)$ for the unbiased walk as the $p \to 1/2$ limit of that
formula, justified by Lemma~\ref{Lemma-4.2}; the uniform distribution is
exact, not an approximation, even though it is obtained as a limit.
Corollary~\ref{corollary-4.5} gives an exact series representation for
any deterministic time-varying schedule, with Algorithm~1 supplying a
finitely truncated approximation with an explicit error bound.
Collectively these establish the interval-expansion approach as a
unified, exact framework for last-coverage analysis on cycle graphs, and
connect foundational probability theory to practical questions in
network design, robotic patrol, and information propagation.

\section*{Disclaimer}

The views expressed in this article are those of the authors and do not reflect the official policy or position of the U.S. Naval Academy, Department of the Navy, the Department of War, or the U.S. Government.

\bibliographystyle{siamplain}
\bibliography{references}

\end{document}